\documentclass[reqno,11pt]{amsart}
\usepackage{amssymb, amsmath}
\usepackage{amsthm, amsfonts,mathrsfs}
\usepackage[T1]{fontenc}
\usepackage{times}
\usepackage{graphicx}
\usepackage{epsfig}
\usepackage{color}
\usepackage[all]{xypic}
\newtheorem{proposition}{Proposition}[section]
\newtheorem{remark}{Remark}[section]

\numberwithin{equation}{section}

\newcommand{\mr}{\mathbb{R}}

\begin{document}
\title[An exact solution for geophysical trapped waves]
{An exact solution for geophysical trapped waves in the presence of an underlying current}%
\author[Lili Fan]{Lili Fan}%
\address[Lili Fan]{College of Mathematics and Information Science,
Henan Normal University, Xinxiang 453007, China}
\email{fanlily89@126.com}
\author[Hongjun Gao$^{\dag}$]{Hongjun Gao$^{\dag}$}
\address[Hongjun Gao]{School of Mathematical Sciences, Institute of Mathematics,
 Nanjing Normal University, Nanjing 210023, China;  \ Jiangsu Center for Collaborative Innovation in Geographical Information Resource Development and Application,  Nanjing Normal University, Nanjing 210023,  China}
\email{gaohj@njnu.edu.cn\, (Corresponding author)}
\author[Qingkun Xiao]{Qingkun Xiao}
\address[Qingkun Xiao]{College of Sciences, Nanjing Agricultural University, Nanjing 210095, China}
\email{xiaoqk@njau.edu.cn}

% --------------------------------------------------------------

\begin{abstract}In this paper we propose an exact and explicit nonlinear solution to the governing equations which retains all the Coriolis terms. In the presence of an underlying current, we seek the trapped waves induced by these solutions in Northern and Southern hemisphere respectively, showing that the retention of the Coriolis force in the governing equations affects significantly the range of the admissible following and adverse currents.
\end{abstract}

\date{}

\maketitle
% -----------------------------------------


\vskip 0.2cm


\section{Introduction}
\large
This paper aims at finding an exact solution for geophysical trapped waves with an underlying current. The trapped waves we consider in this paper are waves for whose amplitude decays rapidly in the meridional direction and geophysical ocean waves are those which take the Coriolis effects on the fluid body induced by the earth's rotation into account, and accordingly the additional terms are brought by the Coriolis force to the governing equations. These governing equations are applicable for a wide range of oceanic and atmospheric flows \cite{CB,Va}. Due to the complexity and intractability of these equations, some simpler approximate models for Equatorial water waves are instigated in recent years to mitigate the Coriolis terms, among which are the so-called $\beta$-plane and $f$-plane approximations.

The $\beta$-plane approximation introduces a linear variation with latitude of the Coriolis parameter to allow for the variation of the Coriolis force from point to point. This approximation applies in regions within $5^\circ$ latitude of the Equator \cite{CB,GS}. On account of the lost of an appreciable level of mathematical detail and structure from the $\beta$-plane approximation as a result of the 'flattening out' of the earth's surface, a number of interesting mathematical models have been recently proposed to retain some of this structure \cite{CoJ15,CoJ16,CoJ,D1,D2}. Whereas the $f$-plane approximation takes a constant Coriolis parameter into account, for which the latitudinal variations are ignored and this approximation has been applied to oceanic flows within a restricted meridional range of approximately $2^\circ$ latitude from the Equator \cite{Co5,CB}.

In this paper, we consider surface water waves propagating zonally in a relatively narrow ocean strip less than a few degrees of latitude wide, where we put no restriction on the latitude, and so the $f$-plane approximation with the Coriolis parameters $f=2\Omega\sin\phi$, $\hat{f}=2\Omega\cos\phi$ being constant is appropriate \cite{CoM}. In consideration of the significant features and complexifications of the underlying currents, both mathematically and physically, in the geophysical dynamics \cite{Co,CoJ15,CoJ2017,CoM,CB,D3,D1,D2,Mo}, we allow for the underlying background currents in the flow. It is hoped that the additional physical complexity brought by the constant Coriolis parameters and underlying currents may be beneficial with respect to potential generalisations in this direction.

To describe the nonlinear dynamic of the given complex fluid flows in detail,
it is remarkable to find an exact solution to the water wave problem. The explicit exact solution of the governing equations for periodic two-dimensional travelling gravity water waves was first discovered by Gerstner \cite{G}, with significant modifications to incorporate geophysical effects along the lines of \cite{Co2}. It is interesting the Gerstner's solution can also be modified to describe edge-waves propagating over a sloping bed \cite{Co1}. Subsequent to \cite{Co1,Co2,G}, a vast and wide-ranging variety of Gerstner-type exact and explicit solutions were derived and analysed to model a number of different physical and geophysical scenarios cf. \cite{Co2,Co3,Co4,Go3,D3,Io,Ma1,Stu1,Stu2} ect. Considering the $f$-plane approximation, Pollard modified the Gerstner's wave solution to describe free surface waves without any restriction on the latitude in a rotating fluid \cite{Po}. Recently, the Pollard's approach has been applied for equatorial waves \cite{Io1} and flows with currents \cite{CoM}, and the instability of the solution has also been investigated in \cite{Io2}.

In sprit of \cite{Co2}, we provide an inherently three-dimensional explicit Gerstner's type solution which incorporates an underlying current to the $f$-plane governing equations, where Constantin and Monismith gave a Pollard's type solution to this water wave problem recently in \cite{CoM}. The solution we obtained  describes in the Lagrangian framework geophysical trapped waves at arbitrary latitude that propagate westwards or eastwards in the presence of a constant underlying background current, which is an extension of the exact solution for equatorial waves in the $f$-plane approximation near the Equator \cite{Hs}. The dispersion relation \eqref{3.15} of this Gerstner's type solution, characterised by the reciprocal Coriolis parameter $\hat{f}$ and the underlying current, has a difference of $O(\Omega^2)$ compared with the dispersion relation of the corresponding Pollard's type solution \cite{CoM}. Another characteristic of the solution is the direct impact induced by the underlying current $c_0$, the Coriolis parameter $f$ and the reciprocal Coriolis parameter $\hat{f}$ on the vorticity. Moreover, we conclude from the investigation of the free-surface interface that the wave motions in such flow admit both the admissible following and adverse currents, for which the range is influenced by the Coriolis parameter $f$ and the dispersion relation. In particular, for water waves propagating zonally near the Equator and covering both the northern and southern hemispheres, this admission allows for both flows with wave speed $c=c_+>0$ or $c=c_-<0$, which is an extension of \cite{D3}.

The remainder of this paper is organized as follows.  In Section 2, we present the governing equations for the geophysical trapped waves with an underlying current. In Section 3, we propose the exact solution to the governing equations with all the Coriolis force retained together with an investigation on the vorticity, the dispersion relation, the free-surface interface and the stratification.
\section{The governing equations}
\large
Concerning surface water waves propagating zonally in a relatively narrow ocean strip less than a few degrees of latitude wide, we take the Coriolis parameters
\begin{equation}\label{2.1}
f=2\Omega\sin\phi,\quad \hat{f}=2\Omega\cos\phi
\end{equation}
as constant. Here $\phi$ represents the latitude and $\Omega=7.29\times10^{-5}$ rad/s is the rotational speed of the Earth. In a reference frame with the origin located at a point fixed on Earth's surface and rotating with the Earth, we consider the zonal coordinate $x$ pointing east, the meridional coordinate $y$ pointing north and the vertical coordinate $z$ pointing up. Then the governing equations in the $f$-plane approximation we solve are given by \cite{CoM,CB,P}
\begin{equation}\label{2.2a}
\begin{cases}
u_t+uu_x+vu_y+wu_z+\hat{f}w-fv&=-\frac 1 \rho P_x,\nonumber\\
v_t+uv_x+vv_y+wv_z+fu&=-\frac 1 \rho P_y,\nonumber\\
w_t+uw_x+vw_y+ww_z-\hat{f}u&=-\frac 1 \rho P_z-g,\nonumber\tag{2.2a}
\end{cases}
\end{equation}
together with the equation of incompressibility
\begin{equation}\label{2.2b}
\nabla  \cdot U=0,\tag{2.2b}
\end{equation}
and with the equation of mass conservation
\begin{equation}\label{2.2c}
\frac {D\rho} {Dt}=0.\tag{2.2c}
\end{equation}
Here $t$ is the time, $U=(u,v,w)$ is the fluid velocity, $g=9.81$ m/s is the gravitational acceleration at the Earth's surface, $\rho$ is the constant density of fluid (a discussion on stratified flows can be seen in Section 3.4), $P$ is the pressure distribution and $\frac {D} {Dt}$ is the material derivative.

The boundary conditions for the fluid on the free-surface $\eta$ are given by
\begin{align}
w&=\eta_t+u\eta_x+v\eta_y,\tag{2.2d}\label{2.2d}\\
P&=P_0,\tag{2.2e} \label{2.2e}
\end{align}
where $P_0$ is the constant atmosphere pressure. The kinematic boundary condition express the fact that the same particles always form the free surface. Finally, we assume the water to be infinitely deep, with the flow converging rapidly with depth to a uniform zonal current, that is,
\begin{equation}\label{2.2f}
(u,v,w) \rightarrow (-c_0,0,0)\;as\;z\rightarrow -\infty.\tag{2.2f}
\end{equation}

\section{Exact solution}
\large
In this section we define an exact solution of the governing equations (2.2). Let us re-express the Euler equation \eqref{2.2a} in the following form
\begin{equation}\label{3.1}
\begin{cases}
\frac {Du} {Dt}+\hat{f}w-fv&=-\frac 1 \rho P_x,\\
\frac {Dv} {Dt}+fu&=-\frac 1 \rho P_y,\\
\frac {Dw} {Dt}-\hat{f}u&=-\frac 1 \rho P_z-g.
\end{cases}
\end{equation}
The Lagrangian framework is adequate for the exact solution. In this framework, the Eulerian coordinates of fluid particles $(x,y,z)$ at time $t$ are expressed as functions of Lagrangian labelling variables $(q,s,r)$, which specify the fluid particle. On the lines of \cite{Co2}, we suppose that the position of a particle at time $t$ is given by
\begin{equation}\label{3.2}
\begin{cases}
x=q-c_0t-\frac 1 k e^{k[r-m(s)]}\sin[k(q-ct)],\\
y=s,\\
z=r+\frac 1 k e^{k[r-m(s)]}\cos[k(q-ct)],
\end{cases}
\end{equation}
where $k$ is the wavenumber and the $c_0$ term represents a constant underlying current such that for $cc_0>0$ the current is adverse, while for $cc_0<0$ the current is following. The expressions of the travelling speed $c$ and the function $m$ depending on $s$ are determined below such that \eqref{3.2} defines an exact solution of the governing equations (2.2). The label domain is given by real values of $(q,s,r)\in(\mr,[-s_0,s_0],(-\infty,r_0])$ such that
\begin{equation}\label{3.3}
r-m(s)\leq r_0<0
\end{equation}
to ensure the flow has the appropriate decay properties.

For notational convenience, let us choose
\begin{equation}\label{3.4}
\xi=k[r-m(s)],\quad \theta=k(q-ct),
\end{equation}
and accordingly the Jacobian matrix of the transformation \eqref{3.2} is given by
\begin{equation}\label{3.5}
J=
\begin{pmatrix}
\frac {\partial x}{\partial q}&\frac {\partial y}{\partial q}&\frac {\partial z}{\partial q}\\
\frac {\partial x}{\partial s}&\frac {\partial y}{\partial s}&\frac {\partial z}{\partial s}\\
\frac {\partial x}{\partial r}&\frac {\partial y}{\partial r}&\frac {\partial z}{\partial r}\\
\end{pmatrix}
=
\begin{pmatrix}
1-e^{\xi}\cos\theta&0&-e^{\xi}\sin\theta\\
m_se^{\xi}\sin\theta&1&-m_se^{\xi}\cos\theta\\
-e^{\xi}\sin\theta&0&1+e^{\xi}\cos\theta
\end{pmatrix}.
\end{equation}
The velocity and acceleration of a particle can also be calculated directly from \eqref{3.2} as
\begin{equation}\label{3.6}
\begin{cases}
u=\frac {Dx} {Dt}=-c_0+c e^{\xi}\cos\theta,\\
v=\frac {Dy} {Dt}=0,\\
w=\frac {Dz} {Dt}=c e^{\xi}\sin\theta,
\end{cases}
\end{equation}
and
\begin{equation}\label{3.7}
\begin{cases}
\frac {Du} {Dt}=kc^2 e^{\xi}\sin\theta,\\
\frac {Dv} {Dt}=0,\\
\frac {Dw} {Dt}=-kc^2 e^{\xi}\cos\theta,
\end{cases}
\end{equation}
respectively. We can therefore write \eqref{3.1} as
\begin{equation}\label{3.8}
\begin{cases}
P_x=-\rho(kc^2 e^{\xi}\sin\theta+\hat{f}ce^{\xi}\sin\theta),\\
P_y=-\rho(-fc_0+fce^{\xi}\cos\theta),\\
P_z=-\rho(-kc^2 e^{\xi}\cos\theta+\hat{f}c_0-\hat{f}ce^{\xi}\cos\theta+g).
\end{cases}
\end{equation}
The change of variables
\begin{equation}\label{3.9}
\begin{pmatrix}
P_q\\P_s\\P_r
\end{pmatrix}
=\begin{pmatrix}
\frac {\partial x}{\partial q}&\frac {\partial y}{\partial q}&\frac {\partial z}{\partial q}\\
\frac {\partial x}{\partial s}&\frac {\partial y}{\partial s}&\frac {\partial z}{\partial s}\\
\frac {\partial x}{\partial r}&\frac {\partial y}{\partial r}&\frac {\partial z}{\partial r}\\
\end{pmatrix}
\begin{pmatrix}
P_x\\P_y\\P_z
\end{pmatrix}\\
\end{equation}
transforms \eqref{3.8} into
\begin{equation}\label{3.10}
\begin{cases}
P_q=-\rho(kc^2 +\hat{f}c-\hat{f}c_0-g)e^{\xi}\sin\theta,\\
P_s=-\rho[m_s(kc^2+\hat{f}c)e^{2\xi}+(fc-\hat{f}c_0m_s-gm_s)e^{\xi}\cos\theta-fc_0 ],\\
P_r=-\rho[-(kc^2+\hat{f}c)e^{2\xi}-(kc^2 +\hat{f}c-\hat{f}c_0-g)e^{\xi}\cos\theta+\hat{f}c_0+g].
\end{cases}
\end{equation}

We prescribe now a suitable pressure function such that \eqref{3.10} holds, proving thus that \eqref{3.2} is indeed an exact solution of the governing equations \eqref{2.2a}. Choosing the expression
\begin{equation}\label{3.11}
m(s)=\frac {fc} {\hat{f}c_0+g} s\geq0
\end{equation}
with
\begin{equation}\label{3.12}
c_0>-\frac g {\hat{f}},
\end{equation}
since otherwise we would have
\begin{equation}\label{3.13}
c_0\leq -\frac g {\hat{f}} \leq -\frac g {2\Omega}\approx-6.7\times 10^{-4} m/s,
\end{equation}
a scenario we can exclude on physical grounds, and accordingly
\begin{align}\label{3.14}
P=&\rho\left[(kc^2 +\hat{f}c-\hat{f}c_0-g)e^{\xi}\cos\theta+\frac {kc^2 +\hat{f}c} {2k}e^{2\xi}+fc_0s-(\hat{f}c_0+g)r\right]\nonumber\\
&+P_0+const,
\end{align}
we observe that equation \eqref{3.10} are satisfied. The consideration of the nonnegativity of $m(s)$ is to ensure a decay in fluid particle motion in the meridional direction (see discussion in Sec 3.3). The pressure must be time independent on the surface due to the dynamic boundary condition \eqref{2.2e}, we thus need to eliminate terms containing $\theta$ in \eqref{3.14}. Therefore
\begin{equation}\label{3.15}
kc^2 +\hat{f}c-\hat{f}c_0-g=0.
\end{equation}
Using \eqref{3.15} in relation \eqref{3.14}, we obtain that the choice of pressure function
\begin{align}\label{3.16}
&P(r,s)=\nonumber\\
&\rho\left[\frac {\hat{f}c_0+g} {2k}e^{2\xi}+fc_0s-(\hat{f}c_0+g)r\right]
+P_0-\rho\left[\frac {\hat{f}c_0+g} {2k}e^{2kr_0}-(\hat{f}c_0+g)r_0\right]
\end{align}
in conjunction with the flow determined by \eqref{3.2} satisfies the governing equations \eqref{2.2a}. The constant terms from the above expression have been chosen to ensure the conditions \eqref{2.2d} and \eqref{2.2e} hold on the free surface.
\subsection{The vorticity}
\large
The inverse of the Jacobian matrix \eqref{3.5} can be computed as
\begin{equation}\label{3.17}
J^{-1}
=\frac 1 {1-e^{2\xi}}
\begin{pmatrix}
1+e^{\xi}\cos\theta&0&e^{\xi}\sin\theta\\
-m_se^{\xi}\sin\theta&1&m_s(e^{\xi}\cos\theta-e^{2\xi})\\
e^{\xi}\sin\theta&0&1-e^{\xi}\cos\theta
\end{pmatrix},
\end{equation}
and thus we can get the velocity gradient tension
\begin{equation}\label{3.18}
\begin{aligned}
\nabla U&=
\begin{pmatrix}
\frac {\partial u}{\partial x}&\frac {\partial u}{\partial y}&\frac {\partial u}{\partial z}\\
\frac {\partial v}{\partial x}&\frac {\partial v}{\partial y}&\frac {\partial v}{\partial z}\\
\frac {\partial w}{\partial x}&\frac {\partial w}{\partial y}&\frac {\partial w}{\partial z}
\end{pmatrix}
=\frac {dJ} {dt} J^{-1}\\
&=\frac {cke^{\xi}} {1-e^{2\xi}}
\begin{pmatrix}
-\sin\theta&m_s(e^{\xi}-\cos\theta)&-e^{\xi}+\cos\theta\\
0&0&0\\
\cos\theta+e^{\xi}&-m_s\sin\theta&\sin\theta
\end{pmatrix},
\end{aligned}
\end{equation}
from which the vorticity $w=(w_y-v_z, u_z-w_x,v_x-u_y)$ is obtained as
\begin{equation}\label{3.19}
\omega=\left(-\frac {kfc^2} {\hat{f}c_0+g} \frac {e^{\xi}\sin\theta} {1-e^{2\xi}},
-\frac {2kce^{2\xi}} {1-e^{2\xi}}, \frac {kfc^2} {\hat{f}c_0+g} \frac{e^{\xi}\cos\theta-e^{2\xi}} {1-e^{2\xi}}\right).
\end{equation}
Now, we give some notations on the vorticity \eqref{3.19}.

\begin{remark}
1. Though the velocity field \eqref{3.6} for the solution \eqref{3.2} is two-dimensional, the vorticity \eqref{3.19} is (weakly) three-dimensional away from the equator (for which $f=0$), with the first and third components depending on the latitude $\phi$. The impact of the latitude $\phi$ on the second component appears implicitly in the dispersion relation \eqref{3.42}.

2. The underlying current $c_0$ features directly in the expression for $\omega$, whereas it plays an implicit role in \cite{D1,D2}.
\end{remark}

\subsection{The dispersion relations}
\large
The relation \eqref{3.15} has implications  in  determining the dispersion relation for the flow by regarding \eqref{3.15} as a quadratic in $c$. To solve it, we enforce
\begin{equation}\label{3.41}
\Delta=\hat{f}^2+4k(\hat{f}c_0+g)> 0,
\end{equation}
which holds automatically by a glance at \eqref{3.12}. The dispersion relation can then be computed as
\begin{equation}\label{3.42}
c_{\pm}=\frac {-\hat{f}\pm\sqrt{\Delta}}{2k},
\end{equation}
featuring contributions from the reciprocal Coriolis parameter $\hat{f}$ and the underlying current $c_0$. If $c=c_+>0$, the wave travels from the west to east and if $c=c_-<0$, it travels from east to west. We now give some remarks on the relations \eqref{3.42}.

\begin{remark}

1. If $c_0=0$, then the dispersion relation reduces to
\begin{equation}\label{3.43}
c_{\pm}=\frac {-\hat{f}\pm\sqrt{\hat{f}^2+4kg}}{2k},
\end{equation}
which corresponds to Gerstner waves that are very slightly modified by rotation. Recalling the dispersion relation of Pollard's waves as \cite{CoM,Po}
\begin{equation}\label{3.44}
(k^2c^2-f^2)c^2=(g-\hat{f}c)^2,
\end{equation}
we obtain the equivalent forms as
\begin{equation}\label{3.45}
k^2c^4-f^2c^2=(g-\hat{f}c)^2
\end{equation}
and
\begin{equation}\label{3.46}
k^2c^4-4\Omega^2c^2+2g\hat{f}c-g^2=0.
\end{equation}
The relation \eqref{3.46} tells us that the Coriolis parameter $f$ does no feature in the expression of the dispersion, which is in accordance with the result \eqref{3.15} in this paper. On the other hand, we get from \eqref{3.15} that in the case $c_0=0$
\begin{equation}\label{3.47}
k^2c^4=(g-\hat{f}c)^2,
\end{equation}
and accordingly the difference of the two dispersion \eqref{3.47} and \eqref{3.45} is $f^2c^2$ ($\sim O(\Omega^2)$), which is relatively much smaller than the terms $\hat{f}c$ ($\sim O(\Omega)$).

2. For equatorial waves we have $f=0$ and $\hat{f}=2\Omega$, so that the constraint \eqref{3.12} reduces to
\begin{equation}\label{3.48}
c_0>-\frac g {2\Omega},
\end{equation}
which agrees with the restriction in \cite{D3,D2}. Besides, the relation \eqref{3.42} reduces to
\begin{equation}\label{3.49}
c_{\pm}=\frac {-\Omega\pm\sqrt{\Omega^2+k(g+2\Omega c_0)}}{k},
\end{equation}
which recovers the result obtained by Constantin in \cite{CoM}, and a further assumption $c_0=0$ leads to the result of Hsu \cite{Hs}.

3. In the case $c_0=c$, we get from \eqref{3.15} that $c=\pm\sqrt{\frac g k}$, which resembles that of both Gerstner's wave and gravity waves in deep-water \cite{Co}. This dispersion relation is also obtained by ignoring the Earth's rotation ($\Omega=0$), in which case the solution \eqref{3.2} effectively reduces to Gerstner's two-dimensional gravity water waves.

\end{remark}
\subsection{The free-surface interface}
\large
By considering the boundary conditions \eqref{2.2d} and \eqref{2.2e}, we now investigate for which values of the current $c_0$ this flow is hydrodynamically possible. This will be achieved by proving that, for each fixed latitude $s$ and $\phi$, there exists a unique solution $r(s)\leq r_0<0$ such that $P(r(s),s)=P_0$ in \eqref{3.16}, which is equivalent to
\begin{equation}\label{3.20}
h(r(s),s)=\frac {\hat{f}c_0+g} {2k}e^{2kr_0}-(\hat{f}c_0+g)r_0,
\end{equation}
where
\begin{equation}\label{3.21}
h(r(s),s)=\frac {\hat{f}c_0+g} {2k}e^{2k[r-\frac{fcs}{\hat{f}c_0+g}]}+fc_0s-(\hat{f}c_0+g)r.
\end{equation}
For $s=0$, we have
\begin{equation}\label{3.22}
h(r,0)=\frac {\hat{f}c_0+g} {2k}e^{2kr}-(\hat{f}c_0+g)r
\end{equation}
and so $r(0)=r_0$. For $s\neq 0$ fixed
\begin{align}\label{3.23}
h_r&=(\hat{f}c_0+g)e^{2k[r-\frac{fcs}{\hat{f}c_0+g}]}-(\hat{f}c_0+g)r\nonumber\\
&=(\hat{f}c_0+g)[e^{2k[r-\frac{fcs}{\hat{f}c_0+g}]}-1]<0
\end{align}
by \eqref{3.3} and \eqref{3.12}. Hence, $h$ is a decreasing function of $r$. Since
\begin{equation}\label{3.24}
\lim_{r\rightarrow-\infty}h(r,s)=\infty,
\end{equation}
if the following inequality holds
\begin{align}\label{3.25}
\lim_{r\rightarrow r_0}h(r,s)
&=\frac{\hat{f}c_0+g}{2k}e^{2k[r_0-\frac{fcs}{\hat{f}c_0+g}]}+fc_0s-(\hat{f}c_0+g)r_0\nonumber\\
&<\frac{\hat{f}c_0+g}{2k}e^{2kr_0}-(\hat{f}c_0+g)r_0,
\end{align}
then the equation \eqref{3.20} has a unique solution. The inequality \eqref{3.25} takes the form
\begin{equation}\label{3.26}
\frac{\hat{f}c_0+g}{2k}e^{2kr_0}\left[e^{-2k\frac{fcs}{\hat{f}c_0+g}}-1\right]+fc_0s
:=A_1(s)+A_2(s)<0,
\end{equation}
where
\begin{equation}\label{3.27}
A_1(s)=\frac{\hat{f}c_0+g}{2k}e^{2kr_0}\left[e^{-2k\frac{fcs}{\hat{f}c_0+g}}-1\right]<0,
\end{equation}
by \eqref{3.11} and \eqref{3.12}. Differentiate \eqref{3.20} with respect to $s$ to get that
\begin{equation}\label{3.28}
h_s=(\hat{f}c_0+g)r'(s)
\left[e^{2k[r-\frac{fcs}{\hat{f}c_0+g}]}-1\right]
+f\left[c_0-ce^{2k[r-\frac{fcs}{\hat{f}c_0+g}]}\right]=0,
\end{equation}
which by the way of \eqref{3.12} gives
\begin{equation}\label{3.29}
r'(s)=\frac f {\hat{f}c_0+g}\left[\frac {c_0-ce^{2k[r-\frac{fcs}{\hat{f}c_0+g}]}} {1-e^{2k[r-\frac{fcs}{\hat{f}c_0+g}]}}\right].
\end{equation}
Bearing in mind that we are seeking geophysical trapped waves, we set
\begin{equation}\label{3.30}
\begin{cases}
r'(s)<0,\quad \text{for}\;s>0,\\
r'(s)>0,\quad \text{for}\;s<0,
\end{cases}
\end{equation}
i.e.
\begin{equation}\label{3.31}
\begin{cases}
f[c_0-ce^{2k[r-\frac{fcs}{\hat{f}c_0+g}]}]<0,\quad \text{for}\;s>0,\\
f[c_0-ce^{2k[r-\frac{fcs}{\hat{f}c_0+g}]}]>0,\quad \text{for}\;s<0,
\end{cases}
\end{equation}
by \eqref{3.3} and \eqref{3.12}. Hence, by $r(0)=r_0$ and \eqref{3.3}, we must necessarily have that
\begin{equation*}
r'(s)-m'(s)=r'(s)-\frac{fc}{\hat{f}c_0+g}
\begin{cases}
<0,\quad \text{for}\;s>0,\\
>0,\quad \text{for}\;s<0,
\end{cases}
\end{equation*}
and accordingly $m(s)=\frac{fcs}{\hat{f}c_0+g}\geq0$ is required. Combining \eqref{3.31}, \eqref{3.26}, \eqref{3.11} with \eqref{3.12}, we obtain the restrictions needed for the hydrodynamical possibility of the flow as
\begin{align}
&fcs>0,\label{3.32}\\
&\frac{\hat{f}c_0+g}{2k}e^{2kr_0}[e^{-2k\frac{fcs}{\hat{f}c_0+g}}-1]+fc_0s
=A_1+A_2<0,\label{3.33}\\
&\begin{cases}
f[c_0-ce^{2k[r-\frac{fcs}{\hat{f}c_0+g}]}]<0,\quad \text{for}\;s>0,\\
f[c_0-ce^{2k[r-\frac{fcs}{\hat{f}c_0+g}]}]>0,\quad \text{for}\;s<0.
\end{cases}\label{3.34}
\end{align}
A detailed discussion on the above restrictions in Northern and Southern hemisphere respectively leads to the following results.
\begin{proposition}\label{pro3.1}
The fluid motion prescribed by \eqref{3.2} represents an exact solution to the governing equations (2.2) in cases that

I. In Northern hemisphere

I-1. The flows with the wave phase speed $c=c_+>0$ admit following currents $-\frac g {\hat{f}}<c_0\leq0$ for $s>0$ and adverse currents $0<c_0<ce^{2kr_0}$ for some $s>0$ close to zero.

I-2. The flows with the wave phase speed $c=c_-<0$ admit following currents $c_0\geq 0$ for $s<0$ and adverse currents $-\frac g {\hat{f}}\leq ce^{2kr_0}<c_0<0$ for $s<0$ close to zero.

II. In Southern hemisphere

II-1. The flows with the wave phase speed $c=c_+>0$ admit following currents $-\frac g {\hat{f}}<c_0\leq0$ for $s<0$ and adverse currents $0<c_0<ce^{2kr_0}$ for $s<0$ close to zero.

II-2. The flows with the wave phase speed $c=c_-<0$ admit following currents $c_0\geq 0$ for $s>0$ and adverse currents $-\frac g {\hat{f}}\leq ce^{2kr_0}<c_0<0$ for $s>0$ close to zero.

This solution represents three-dimensional, nonlinear geophysical trapped waves. The free surface $z=\eta(x,y,t)$ is implicitly prescribed at $s=0$ by setting $r=r_0$ in \eqref{3.2}, and for other fixed latitude $s\in[-s_0,0]$ or $s\in[0,s_0]$, there exists a unique value $r(s)<r_0$ which implicitly prescribes the free surface $z=\eta(x,s,t)$ by waves of setting $r=r(s)$ in \eqref{3.2}.
\end{proposition}
\begin{proof}
We just discuss the Case I as the Case II can be treated in a similar way. In the Northern hemisphere, we have $f>0$.

1. For $c=c_+>0$, we obtain from \eqref{3.32} that $s>0$ and from \eqref{3.33}-\eqref{3.34} that
\begin{align}
&\frac{\hat{f}c_0+g}{2k}e^{2kr_0}\left[e^{-2k\frac{fcs}{\hat{f}c_0+g}}-1\right]+fc_0s
=A_1(s)+A_2(s)<0,\quad \text{for}\;s>0, \label{3.35}\\
&f\left[c_0-ce^{2k[r-\frac{fcs}{\hat{f}c_0+g}]}\right]<0,\quad \text{for}\;s>0.
\label{3.36}
\end{align}
The above conditions are satisfied obviously for $-\frac g {\hat{f}}<c_0\leq0$. In terms of $c_0>0$, it can be seen that the condition \eqref{3.35} will break down for large enough values of $s>0$. Noting that we are interested in water waves propagating in a relatively narrow ocean strip, it is possible for \eqref{3.35} to hold for a small value $s$ depending on the size of the current $c_0>0$. Since $A_1(0)+A_2(0)=0$, we must necessarily have
\begin{equation}\label{3.37}
A_1'(s)+A_2'(s)=f(c_0-ce^{2k[r_0-\frac{fcs}{\hat{f}c_0+g}]})<0
\end{equation}
for $s>0$ close to zero to ensure the validity of \eqref{3.35}. For a given $0<c_0<ce^{2kr_0}$, \eqref{3.37} holds for some $s\in(0,s_1]$ and accordingly \eqref{3.36} holds for $s\in(0,s_0]$ with $s_0<s_1$.

2. For $c=c_-<0$, we get from \eqref{3.32} that $s<0$ and \eqref{3.33}-\eqref{3.34} reduces to
\begin{align}
&\frac{\hat{f}c_0+g}{2k}e^{2kr_0}\left[e^{-2k\frac{fcs}{\hat{f}c_0+g}}-1\right]+fc_0s
=A_1(s)+A_2(s)<0, \quad \text{for}\;s<0, \label{3.38}\\
&f\left[c_0-ce^{2k[r-\frac{fcs}{\hat{f}c_0+g}]}\right]>0,\quad \text{for}\;s<0.
\label{3.39}
\end{align}
The above conditions are satisfied obviously for $c_0\geq0$. In terms of $-\frac g {\hat{f}}<c_0<0$, it can be seen that the condition \eqref{3.38} will break down for negative enough values of $s<0$. The consideration about water waves propagating in a relatively narrow ocean strip makes it possible for \eqref{3.38} to hold for a small value $s$ depending on the size of the current $c_0<0$. Since $A_1(0)+A_2(0)=0$, we must still necessarily have
\begin{equation}\label{3.40}
A_1'(s)+A_2'(s)=f(c_0-ce^{2k[r_0-\frac{fcs}{\hat{f}c_0+g}]})>0
\end{equation}
for $s<0$ close to zero to ensure the validity of \eqref{3.38}. For a given $0>c_0>ce^{2kr_0}\geq-\frac g {\hat{f}}$, \eqref{3.40} holds for some $s\in[-s_1,0)$ and accordingly \eqref{3.39} holds for $s\in[-s_0,0)$ with $s_0<s_1$. This completes the proof of Case I.
\end{proof}

\begin{remark} For water waves propagating zonally near the Equator and covering both the northern and southern hemispheres, where we assume $\sin\phi\approx\phi\neq0$, the conclusion of I-1 and II-1 for $c=c_+>0$ is in accordance with the result for equatorial geophysical water waves with underlying current by Henry in \cite{D3}, whereas the solutions we constructed here admit flows with wave speed $c=c_-<0$.
\end{remark}

\subsection{Stratification}
In the absence of an underlying current $c_0=0$, we can accommodate a stratified fluid through assuming that the density has a steady function dependence of the form $\rho(x,y,z,t)=\rho(x-ct,y,z)$. The equation of mass conservation \eqref{2.2c} recasts to
\begin{equation}\label{3.50}
(u-c)\rho_x+w\rho_z=c\left(\rho_x(e^\xi\cos\theta-1)+\rho_ze^\xi\sin\theta\right)=0,
\end{equation}
and a direct computation leads to
\begin{equation}\label{3.51}
\rho_q=\rho_x\frac {\partial x}{\partial q}+\rho_y \frac {\partial y}{\partial q}
+\rho_z\frac {\partial z}{\partial q}
=\rho_x(1-e^{\xi}\cos\theta)-\rho_ze^{\xi}\sin\theta=0.
\end{equation}
Therefore the density $\rho$ is independent of $q$. Defining the density  function by
\begin{equation}\label{3.52}
\rho(r,s)=F(\frac {e^{2\xi}} {2k}-r),
\end{equation}
where $F:(0,\infty)\rightarrow (0, \infty)$ is continuously differentiable and non-decreasing and the pressure function \eqref{3.16} is suitably adapted, for $\mathcal{F}'=F$ with $\mathcal{F}(0)=0$, the function
\begin{equation}\label{3.53}
P=g\mathcal{F}(\frac {e^{2\xi}} {2k}-r)+P_0-g\mathcal{F}(\frac {e^{2kr_0}} {2k}-r_0).
\end{equation}

\vspace{0.5cm}
\noindent {\bf Acknowledgements.}
The work of Fan is supported by a NSFC Grant No. 11701155, NSF of Henan Normal University, Grant No. 2016QK03, and the Startup Foundation for Introducing Talent of Henan Normal University, Grant No. qd16150. The work of Gao is partially supported  by the NSFC grant No. 11531006 and PAPD of Jiangsu Higher Education Institutions. The work of Xiao is partially supported by the Fundamental Research Funds for the Central Universities (Grant No. KYZ201748).

\end{document}